\documentclass{amsart}%
\usepackage{amsmath}
\usepackage{amssymb}
\usepackage{amsthm}
\usepackage{amsfonts}
\usepackage{amssymb}
\usepackage{graphicx}%
\providecommand{\U}[1]{\protect\rule{.1in}{.1in}}
\providecommand{\U}[1]{\protect \rule{.1in}{.1in}}
\providecommand{\U}[1]{\protect \rule{.1in}{.1in}}
\newtheorem{theorem}{Theorem}
\theoremstyle{plain}

\newtheorem{corollary}{Corollary}

\newtheorem{definition}{Definition}
\newtheorem{example}{Example}

\newtheorem{lemma}{Lemma}

\newtheorem{proposition}{Proposition}

\numberwithin{equation}{section}
\begin{document}

\begin{center}
\bigskip

\noindent{\Large \textbf{Approximately Groups in Proximal Relator Spaces: An Algebraic View of Digital Images}%
}\\[15pt]{\large Ebubekir \.{I}nan}
\\[5mm]

\begin{center}
Department of Mathematics, Faculty of Arts and Sciences, Ad\i yaman University, Ad\i yaman, Turkey.\\
Computational Intelligence Laboratory, Department of Electrical \& Computer Engineering, University of Manitoba, Winnipeg, Manitoba, Canada.
\end{center}

\bigskip

\begin{center}
\textit{einan@adiyaman.edu.tr}
\end{center}

\bigskip

\textbf{Abstract} \\[1mm]
\end{center}

The focus of this article is to define the descriptively approximations in
proximal relator spaces. Afterwards, descriptive approximately
algebraic structures such as groupoids, semigroups and groups in digital images endowed with descriptive Efremovi\u{c} proximity relations were introduced. \\[3mm]



\section{Introduction}

In the concept of ordinary algebraic structures, such a structure that
consists of a nonempty set of abstract points with one or more binary
operations, which are required to satisfy certain axioms. For example, a
groupoid is an algebraic structure $\left(  A,\circ\right)  $ consisting of a
nonempty set $A$ and a binary operation \textquotedblleft$\circ$%
\textquotedblright\ defined on $A$ \cite{Clifford1961}. In a groupoid, the
binary operation \textquotedblleft$\circ$\textquotedblright\ must be only
closed in $A$, i.e., for all $a,b$ in $A$, the result of the operation $a\circ
b$ is also in $A$. As for the proximal relator space, the sets are
composed of non-abstract points instead of abstract points. And these points
describable with feature vectors in proximal relator spaces. Descriptively upper approximation of a nonempty set is obtained by using the set of points composed by the proximal relator space together with
matching features of points. In the algebraic structures constructed on proximal relator spaces, the basic tool is consideration of descriptively upper approximations
of the subsets of non-abstract points. In a groupoid $A$ in proximal relator space, the binary operation \textquotedblleft$\circ
$\textquotedblright\ may be closed in descriptively upper approximation of $A$, i.e., for
all $a,b$ in $A$, $a\circ b$ is in descriptively upper approximation of $A$. In 2012 and 2014, \.{I}nan et al investigated similar view of this subject \cite{inan2012,inan2014,ozturk2014}.

There are two important differences between ordinary algebraic structures and
descriptive approximately algebraic structures. The first one is working with non-abstract
points such as digital images while the second one is considering of descriptively upper approximations of the
subsets of non-abstract points for the closeness of binary operations.

Essentially, the aim of this paper is to obtain algebraic structures such as descriptive approximately groupoids using sets and operations that ordinary are not being algebraic structures. Moreover two examples with working on digital images endowed with descriptive proximity relation were given.

\section{Preliminaries}

A \emph{relator} is a nonvoid family of relations $\mathcal{R}$ on a nonempty
set $X$. The pair $(X,\mathcal{R})$ (also denoted $X(\mathcal{R})$) is called
a relator space. Relator spaces are natural generalisations of ordered sets
and uniform spaces~\cite{Szaz1987}. With the introduction of a family of
proximity relations on $X$, a proximal relator space $(X,\mathcal{R}%
_{\delta})$ ($X(\mathcal{R}_{\delta})$) was obtained.  As in~\cite{Peters2015relator}, $(\mathcal{R}_{\delta})$ contains proximity relations, namely, Efremovi\u{c} proximity $\delta$~\cite{efremovic1951,Efremovic1952}, Lodato proximity, Wallman proximity, descriptive proximity $\delta_{\Phi}$ in defining $\mathcal{R}_{\delta_{\Phi}}$~\cite{Peters2012ams,Peters2013mcsintro,Peters2014adyu}. 

Proximity space axiomatized by V. A. Efremovi\u{c} in 1934 under the name of infinitesimal space and published in 1951. An Efremovi\u{c} proximity $\delta$ is a relation on $2^X$ that satisfies
\begin{itemize}
\item [1$^o$] $A\ \delta\ B \Rightarrow B\ \delta\ A$.
\item [2$^o$] $A\ \delta\ B \Rightarrow  A \neq \emptyset$ and $B \neq \emptyset$.
\item [3$^o$] $A \cap B \neq \emptyset \Rightarrow A\ \delta\ B$.
\item [4$^o$] $A\ \delta\ (B \cup C)$ if and only if $A\ \delta\ B$ or $A\ \delta\ C$.
\item [5$^o$] $\{x\}\ \delta\ \{y\}$ if and only if $x = y$.
\item [6$^o$] \label{axiom:EF}{EF axiom}. $A\ \underline{\delta}\ B \Rightarrow \exists E\subseteq X$ such that $A\ \underline{\delta}\ E$ and $E^c\ \underline{\delta}\ B$.
\end{itemize}
Lodato proximity~\cite{Lodato1962,Lodato1964,Lodato1966} swaps the EF axiom~\ref{axiom:EF} for the following condition:
\[
A\ \delta\ B \  \textrm{and} \ \forall b \in B, \{b\}\ \delta\ C \Rightarrow A\ \delta\ C.\ (Lodato Axiom)
\]
In this article, only two
proximity relations, namely, the Efremovi\u{c} proximity $\delta
$~\cite{Efremovic1952} and the descriptive proximity $\delta_{\Phi}$ in
defining a descriptive proximal relator space (denoted by $(X,\mathcal{R}%
_{\delta_{\Phi}})$) were used.

In a discrete space, a non-abstract point has a location and features that can
be measured~\cite[\S 3]{Kovar2011}. Let $X$ be a nonempty set of non-abstract
points in a proximal relator space $(X,\mathcal{R}_{\delta_{\Phi}})$ and let
$\Phi=\left\{  \phi_{1},\dots,\phi_{n}\right\}  $ a set of probe functions
that represent features of each $x\in X$.
%For example, leads to a proximal view of sets of picture points in digital images~\cite{Peters2013mcsintro}.
A \emph{probe function} $\Phi:X\rightarrow\mathbb{R}$ represents a feature of
a sample point in a picture. Let $\Phi(x)=(\phi_{1}(x),\dots,\phi_{n}(x))$
denote a feature vector for $x$, which provides a description of each $x\in
X$. To obtain a descriptive proximity relation (denoted by $\delta_{\Phi}$),
one first chooses a set of probe functions. Let $A,B\in2^{X}$ and
$\mathcal{Q}(A),\mathcal{Q}(B)$ denote sets of descriptions of points in
$A,B$, respectively. For example, $\mathcal{Q}(A)=\left\{  \Phi(a) \mid a\in
A\right\}  $. The expression $A\ \delta_{\Phi}\ B$ reads $A\ $is descriptively
near$\ B$. Similarly, $A\ \underline{\delta}_{\Phi}\ B$ reads $A\ $is
descriptively far from$\ B$. The descriptive proximity of $A$ and $B$ is
defined by
\[
A\ \delta_{\Phi}\ B\Leftrightarrow\mathcal{Q}(A)\cap\mathcal{Q}(B)\neq
\emptyset\text{.}%
\]

The relation $\delta_{\Phi}$ is called a \emph{descriptive proximity
relation}. Similarly, $A\ \underline{\delta}_{\Phi}\ B$ denotes that $A$ is
descriptively far (remote) from $B$.

$Object$ $Description:$ $\Phi\left(  x\right)  =\left(  \varphi_{1}\left(
x\right)  ,\varphi_{2}\left(  x\right)  ,\varphi_{3}\left(  x\right)
,...,\varphi_{i}\left(  x\right)  ,...,\varphi_{L}\left(  x\right)  \right)  $.

The intuition underlying a description $\Phi\left(  x\right)  $ is a recording
of measurements from sensors, where each sensor is modelled by a function
$\varphi_{i}$.%

\begin{tabular}
[c]{l}
\end{tabular}

\begin{definition}
\textbf{(}Set Description\textbf{, }\cite{Naimpally2012}\textbf{)} Let
$X$ be a nonempty set of non-abstract points, $\Phi$ an object description and
$A$ a subset of $X$. Then the \textit{set description} of $A$ is
defined as%
\[
\mathcal{Q}(A)=\{ \Phi(a)\mid a\in A\} \text{.}%
\]

\end{definition}

\begin{definition}
\textbf{(}Descriptive Set Intersection\textbf{, }\cite{Naimpally2012,
Peters2012ams}\textbf{)} Let $X$ be a nonempty set of non-abstract points, $A$ and
$B$ any two subsets of $X$. Then the descriptive (set) intersection
of $A$ and $B$ is defined as%
\[
A\underset{\Phi}{\cap}B=\left\{  x\in A\cup B\mid\Phi\left(  x\right)
\in\mathcal{Q}\left(  A\right)  \text{ }and\text{ }\Phi\left(  x\right)
\in\mathcal{Q}\left(  B\right)  \right\}  \text{.}%
\]

\end{definition}

\begin{definition}
\cite{ PetersMCSintro} Let $X$ be a nonempty set of non-abstract points, $A$ and $B$
any two subsets of $X$. If $\mathcal{Q}(A)\cap\mathcal{Q}%
(B)\neq\emptyset$, then $A$ is called descriptively near $B$ and denoted by
$A\delta_{\Phi}B$.
\end{definition}

\begin{definition}
\textbf{(}Descriptive Nearness Collections\textbf{,} \cite{ PetersMCSintro}\textbf{)
}Let $X$ be a nonempty set of non-abstract points and $A$ any subset of
$X$. Then the descriptive nearness collection $\xi_{\Phi}\left(
A\right)  $ is defined by%
\[
\xi_{\Phi}\left(  A\right)  =\left\{  B\in\mathcal{P}\left(  X%
\right)  \mid A\delta_{\Phi}B\right\}  \text{.}%
\]

\end{definition}

\begin{theorem}
\label{Th0}\cite{ PetersMCSintro} Let $\Phi$ be an object description, $A$ any subset
of $X$ and $\xi_{\Phi}\left(  A\right)  $ a descriptive nearness
collections. Then $A\in\xi_{\Phi}\left(  A\right)  $.
\end{theorem}

\section{\bigskip Descriptively Approximations}

Let $\left(  X,\mathcal{R}_{\delta_{\Phi}}\right)  $ be descriptive proximal
relator space and $A\subset X$, where $A$ contains non-abstract objects. A descriptive closure of a point $a\in A$ is
defined by
\[
cl_{\Phi}\left(  a\right)  =\left\{  x\in X \mid \Phi\left(  a\right)  =\Phi\left(
x\right)  \right\}  \text{.}%
\]

Descriptively lower approximation of the set $A$\ begins by determining which
descriptive closure $cl_{\Phi}\left(  a\right)  $ are subsets of \textit{set
description} $\mathcal{Q}\left(  A\right)  $. This discovery process leads to
the construction of what is known as the descriptively lower approximation of
$A\subseteq X$, which is denoted by $\Phi_{\ast}A$.

\begin{definition}
(Descriptively Lower Approximation of a Set) Let $\left(  X,\mathcal{R}%
_{\delta_{\Phi}}\right)  $ be descriptive proximal relator space and $A\subset
X$. A descriptively lower
approximation of $A$ is defined as%
\[
\Phi_{\ast}A=\{a\in A \mid cl_{\Phi}\left(  a\right)  \subseteq\mathcal{Q}\left(
A\right)  \}\text{.}%
\]

\end{definition}

\begin{definition}
\label{Df4.1}(Descriptively Upper Approximation of a Set) Let $\left(
X,\mathcal{R}_{\delta_{\Phi}}\right)  $ be descriptive proximal relator space
and $A\subset X$. A descriptively
upper approximation of $A$ is defined as
\[
\Phi^{\ast}A=\{x\in X \mid x \delta_{\Phi} A\}\text{.}
\]
\end{definition}

\begin{definition}
(Descriptively Boundary Region) Let $Bnd_{\Phi}A$ denote the descriptively
boundary region of a set $A\subseteq X$ defined by
\[
\Phi_{Bnd}A=\Phi^{\ast}A\setminus\Phi_{\ast}A=\left\{  x\mid x\in\Phi^{\ast
}A\text{ \textit{and} }x\notin\Phi_{\ast}A\right\}  \text{.}%
\]
\end{definition}

\begin{lemma}
\label{Lem0}Let $\left(  X,\mathcal{R}_{\delta_{\Phi}}\right)  $ be
descriptive proximal relator space and $A,B\subset X$, then

(i) $\mathcal{Q}\left(  A\cap B\right)  =\mathcal{Q}\left(  A\right)
\cap\mathcal{Q}\left(  B\right)  $,

(ii) $\mathcal{Q}\left(  A\cup B\right)  =\mathcal{Q}\left(  A\right)
\cup\mathcal{Q}\left(  B\right)  $.
\end{lemma}

\begin{theorem}
Let $\left(  X,\mathcal{R}_{\delta_{\Phi}}\right)  $ be descriptive proximal
relator space and $A,B\subset X$,  then the following statements hold.

(1) $\left(  \Phi_{\ast}A\right)  \subseteq A\subseteq\left(  \Phi^{\ast
}A\right)  $,

(2) $\Phi^{\ast}\left(  A\cup B\right)  =\left(  \Phi^{\ast}A\right)
\cup\left(  \Phi^{\ast}B\right)  $,

(3) $\Phi_{\ast}\left(  A\cap B\right)  =\left(  \Phi_{\ast}A\right)
\cap\left(  \Phi_{\ast}B\right)  $,

(4) If $A\subseteq B$, then $\left(  \Phi_{\ast}A\right)  \subseteq\left(
\Phi_{\ast}B\right)  $,

(5) If $A\subseteq B$, then $\left(  \Phi^{\ast}A\right)  \subseteq\left(
\Phi^{\ast}B\right)  $,

(6) $\Phi_{\ast}\left(  A\cup B\right)  \supseteq\left(  \Phi_{\ast}A\right)
\cup\left(  \Phi_{\ast}B\right)  $,

(7) $\Phi^{\ast}\left(  A\cap B\right)  \subseteq\left(  \Phi^{\ast}A\right)
\cap\left(  \Phi^{\ast}B\right)  $.
\end{theorem}

\begin{proof}
(1) Let $a\in\left(  \Phi_{\ast}A\right)  $, then $cl_{\Phi}\left(  a\right)
\subseteq\mathcal{Q}\left(  A\right)  $, where $a\in A$. Hence $\left(
\Phi_{\ast}A\right)  \subseteq A$. Let $a\in A$ and it is obvious that
$\Phi\left(  a\right)  \in\mathcal{Q}\left(  A\right)  $, i.e. $\Phi\left(  a\right) \cap\mathcal{Q}\left(  A\right) \neq \emptyset $  and so $a\delta_{\Phi} A$. Therefore $a\in\left(  \Phi^{\ast}A\right)  $ and then $A\subseteq\left(  \Phi^{\ast
}A\right)  $.

(2)%
\begin{tabular}
[c]{l}
\end{tabular}

\begin{tabular}
[c]{ll}%
$x\in\Phi^{\ast}\left(  A\cup B\right)  $ & $ \Leftrightarrow x\delta_{\Phi} \left(A\cup B \right)   $\\
& $ \Leftrightarrow \Phi\left(  x\right) \cap\mathcal{Q}\left(  A\cup B\right) \neq \emptyset $\\
& $\Leftrightarrow\Phi\left(
x\right)  \in\mathcal{Q}\left(  A\cup B\right)  $, from Lemma \ref{Lem0}\\
& $\Leftrightarrow\Phi\left(  x\right)  \in\mathcal{Q}\left(  A\right)
\cup\mathcal{Q}\left(  B\right)  $\\
& $\Leftrightarrow\Phi\left(  x\right)  \in\mathcal{Q}\left(  A\right)  $ or
$\Phi\left(  x\right)  \in\mathcal{Q}\left(  B\right)  $\\
& $ \Leftrightarrow \Phi\left(  x\right) \cap\mathcal{Q}\left(  A\right) \neq \emptyset $ or $ \Phi\left(  x\right) \cap\mathcal{Q}\left(  B\right) \neq \emptyset $\\
& $ \Leftrightarrow x\delta_{\Phi} \left(A \right) $ or $ x\delta_{\Phi} \left(B \right) $ \\
& $\Leftrightarrow x\in\left(  \Phi^{\ast}A\right)  $ or $x\in\left(
\Phi^{\ast}B\right)  $\\
& $\Leftrightarrow x\in\left(  \Phi^{\ast}A\right)  \cup\left(  \Phi^{\ast
}B\right)  $.
\end{tabular}

Hence $\Phi^{\ast}\left(  A\cup B\right)  =\left(  \Phi^{\ast}A\right)
\cup\left(  \Phi^{\ast}B\right)  $.

(3)%
\begin{tabular}
[c]{l}
\end{tabular}

\begin{tabular}
[c]{ll}%
$x\in\Phi_{\ast}\left(  A\cap B\right)  $ & $\Leftrightarrow cl_{\Phi}\left(
x\right)  \subseteq\mathcal{Q}\left(  A\cap B\right)  $, $x\in A\cap B$, from
Lemma \ref{Lem0}\\
& $\Leftrightarrow cl_{\Phi}\left(  x\right)  \subseteq\mathcal{Q}\left(
A\right)  \cap\mathcal{Q}\left(  B\right)  $\\
& $\Leftrightarrow cl_{\Phi}\left(  x\right)  \subseteq\mathcal{Q}\left(
A\right)  $ and $cl_{\Phi}\left(  x\right)  \subseteq\mathcal{Q}\left(
B\right)  $\\
& $\Leftrightarrow x\in\left(  \Phi_{\ast}A\right)  $ and $x\in\left(
\Phi_{\ast}B\right)  $\\
& $\Leftrightarrow x\in\left(  \Phi_{\ast}A\right)  \cap\left(  \Phi_{\ast
}B\right)  $.
\end{tabular}

Thus $\Phi_{\ast}\left(  A\cap B\right)  =\left(  \Phi_{\ast}A\right)
\cap\left(  \Phi_{\ast}B\right)  $.

(4) Let $A\subseteq B$, then $A\cap B=A$. From statement (3) we have
$\Phi_{\ast}A=\Phi_{\ast}\left(  A\cap B\right)  =\left(  \Phi_{\ast}A\right)
\cap\left(  \Phi_{\ast}B\right)  $. Hence $\left(  \Phi_{\ast}A\right)
\subseteq\left(  \Phi_{\ast}B\right)  $.

(5) Let $A\subseteq B$, then $A\cup B=B$. From statement (2) we get
$\Phi^{\ast}B=\Phi^{\ast}\left(  A\cup B\right)  =\left(  \Phi^{\ast}A\right)
\cup\left(  \Phi^{\ast}B\right)  $. This implies that $\left(  \Phi^{\ast
}A\right)  \subseteq\left(  \Phi^{\ast}B\right)  $.

(6) Since $A\subseteq A\cup B$ and $B\subseteq A\cup B$, by (4) we have
$\left(  \Phi_{\ast}A\right)  \subseteq\Phi_{\ast}\left(  A\cup B\right)  $
and $\left(  \Phi_{\ast}B\right)  \subseteq\Phi_{\ast}\left(  A\cup B\right)
$.

Hence $\left(  \Phi_{\ast}A\right)  \cup\left(  \Phi_{\ast}B\right)
\subseteq\Phi_{\ast}\left(  A\cup B\right)  $.

(7) We know that $A\cap B\subseteq A$ and $A\cap B\subseteq B$. From statement
(5) we have $\Phi^{\ast}\left(  A\cap B\right)  \subseteq\left(  \Phi^{\ast
}A\right)  $ and $\Phi^{\ast}\left(  A\cap B\right)  \subseteq\left(
\Phi^{\ast}B\right)  $.

Thus $\Phi^{\ast}\left(  A\cap B\right)  \subseteq\left(  \Phi^{\ast}A\right)
\cap\left(  \Phi^{\ast}B\right)  $.
\end{proof}

\bigskip

\section{\bigskip Approximately Groups in Descriptive Proximal Relator Spaces}

\begin{definition}
Let $\left(  X,\mathcal{R}_{\delta_{\Phi}}\right)  $ be descriptive proximal
relator space and let \textquotedblleft$\cdot$\textquotedblright\ a binary
operation defined on $X$. A subset $G$ of the set of $\ X$\ is called a
\textit{descriptive approximately} groupoid in descriptive proximal relator space if
\ $x\cdot y\in\Phi{}^{\ast}G$, for all $x,y\in G$.
\end{definition}

\begin{definition}
Let $\left(  X,\mathcal{R}_{\delta_{\Phi}}\right)  $ be descriptive proximal
relator space and let "$\cdot$" a binary operation defined on $X$. A subset
$G$ of the set of $\ X$\ is called a \textit{descriptive approximately group} in
descriptive proximal relator space if the following properties are satisfied:

\begin{enumerate}
\item[$(\mathcal{A}G_{1})$] For all $x,y\in G$, $x\cdot y\in\Phi{}^{\ast}G$,

\item[$(\mathcal{A}G_{2})$] For all $x,y,z\in G$, $\left(  x\cdot y\right)  \cdot
z=x\cdot\left(  y\cdot z\right)  $ property\ holds in $\Phi{}^{\ast}G$,

\item[$(\mathcal{A}G_{3})$] There exists $e\in\Phi{}^{\ast}G$ such that $x\cdot e=e\cdot
x=x$ for all $x\in G$\emph{ }($e$ is called the \textit{approximately identity element}
 of $G$),

\item[$(\mathcal{A}G_{4})$] There exists $y\in G$ such that $x\cdot y=y\cdot x=e$ for
all $x\in G$ ($y$ is called the \textit{inverse} of $x$ in
$G$ and denoted as $x^{-1}$).
\end{enumerate}
\end{definition}

A subset $S$ of the set of$\ X$\ is called a \textit{descriptive approximately
semigroup} in descriptive proximal relator space if \\ 
$\left(  \mathcal{A}S_{1}\right)  $
$x\cdot y\in\Phi{}^{\ast}S$, for all $x,y\in S$ and\\ 
$\left(  \mathcal{A}S_{2}\right)  $
$\left(  x\cdot y\right)  \cdot z=x\cdot\left(  y\cdot z\right)  $
property\ holds in $\Phi{}^{\ast}S$, for all $x,y,z\in S$\\ properties are satisfied.

If descriptive approximately semigroup have an approximately
identity element $e\in\Phi{}^{\ast}S$ such that $x\cdot e=e\cdot x=x$ for all
$x\in S$, then $S$ is called a \textit{descriptive approximately monoid} in
descriptive proximal relator space.

If $x\cdot y=y\cdot x$, for all $x,y\in S$ property holds in $\Phi{}^{\ast}G$,
then $G$ is \textit{commutative descriptive approximately groupoid,
semigroup, group or monoid} in descriptive proximal relator space.
\begin{example}
Let $X$ be a digital image endowed with descriptive proximity relation
$\delta_{\Phi}$ and consists of $25$ pixels as in Fig. \ref{fig:pixels1}.

\begin{figure}[!ht]
	\centering
		\includegraphics[scale=0.3]{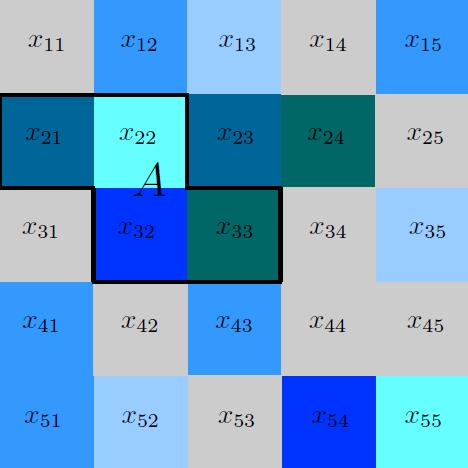}
	\caption{Digital image}
	\label{fig:pixels1}
\end{figure}

A pixel $x_{ij}$ is an element at position $\left(  i,j\right)  $ (row and
column) in digital image $X$. Let $\phi$ be a probe function that represent
RGB\ colour of each pixel are given in Table 1.

\footnotesize
{
\[%
\begin{tabular}
[c]{l|ccc}
& $Red$ & $Green$ & $Blue$\\\hline
$x_{11}$ & $204$ & $204$ & $204$\\
$x_{12}$ & $51$ & $153$ & $255$\\
$x_{13}$ & $204$ & $255$ & $255$\\
$x_{14}$ & $204$ & $204$ & $204$\\
$x_{15}$ & $51$ & $153$ & $255$\\
$x_{21}$ & $0$ & $102$ & $153$\\
$x_{22}$ & $102$ & $255$ & $255$\\
$x_{23}$ & $0$ & $102$ & $153$\\
$x_{24}$ & $0$ & $102$ & $102$\\
$x_{25}$ & $204$ & $204$ & $204$\\
$x_{31}$ & $204$ & $204$ & $204$\\
$x_{32}$ & $0$ & $51$ & $255$\\
$x_{33}$ & $0$ & $102$ & $102$%
\end{tabular}
\ \text{\ \ \ }%
\begin{tabular}
[c]{lccc}
& \multicolumn{1}{|c}{$Red$} & $Green$ & $Blue$\\\hline
$x_{34}$ & \multicolumn{1}{|c}{$204$} & $204$ & $204$\\
$x_{35}$ & \multicolumn{1}{|c}{$204$} & $255$ & $255$\\
$x_{41}$ & \multicolumn{1}{|c}{$51$} & $153$ & $255$\\
$x_{42}$ & \multicolumn{1}{|c}{$204$} & $204$ & $204$\\
$x_{43}$ & \multicolumn{1}{|c}{$51$} & $153$ & $255$\\
$x_{44}$ & \multicolumn{1}{|c}{$204$} & $204$ & $204$\\
$x_{45}$ & \multicolumn{1}{|c}{$204$} & $204$ & $204$\\
$x_{51}$ & \multicolumn{1}{|c}{$51$} & $153$ & $255$\\
$x_{52}$ & \multicolumn{1}{|c}{$204$} & $255$ & $255$\\
$x_{53}$ & \multicolumn{1}{|c}{$204$} & $204$ & $204$\\
$x_{54}$ & \multicolumn{1}{|c}{$0$} & $51$ & $255$\\
$x_{55}$ & \multicolumn{1}{|c}{$102$} & $255$ & $255$\\
&  &  &
\end{tabular}
\]
}
\normalsize
\begin{center}
Table 1.
\end{center}
Let
\[%
\begin{tabular}
[c]{ll}%
$\cdot:X\times X$ & $\longrightarrow X$\\
\multicolumn{1}{r}{$\left(  x_{ij},x_{kl}\right)  $} & $\longmapsto
x_{ij}\cdot x_{kl}=x_{pr},$ $p=\min\left\{  i,k\right\}  $ and $r=\min\left\{
j,l\right\}  $%
\end{tabular}
\]
be a binary operation on $X$ and $A=\left\{  x_{21},x_{22},x_{32}%
,x_{33}\right\}  $ a subimage (subset) of $X$. 

We can compute the descriptively upper approximation of $A$ by using the
Definition \ref{Df4.1}. $\Phi^{\ast}A=\{x_{ij}\in X \mid x_{ij} \delta_{\phi} A  \}$, where $\mathcal{Q}(A)=\{\phi(x_{ij})\mid
x_{ij}\in A\}$. Then $\phi\left(  x_{ij}\right)
\cap \mathcal{Q}\left(  A\right)\neq \emptyset$\ such that $x_{ij}\in X$. From Table 1, we obtain%

\begin{tabular}
[c]{ll}%
$\mathcal{Q}(A)$ & $=\left\{  \phi\left(  x_{21}\right)  ,\phi\left(
x_{22}\right)  ,\phi\left(  x_{32}\right)  ,\phi\left(  x_{33}\right)
\right\}  $\\
& $=\left\{  \left(  0,102,153\right)  ,\left(  102,255,255\right)  ,\left(
0,51,255\right)  ,\left(  0,102,102\right)  \right\}  $.
\end{tabular}

Hence we get $\Phi^{\ast}A=\left\{  x_{21},x_{22},x_{23},x_{24},x_{32}%
,x_{33},x_{54},x_{55}\right\}$ as shown in Fig. \ref{fig:pixels2u}.

\begin{figure}[!ht]
	\centering
		\includegraphics[scale=0.3]{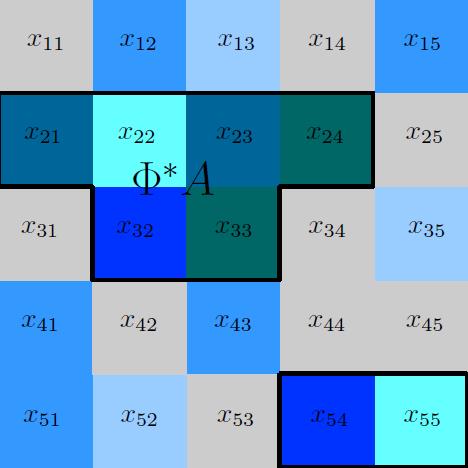}
	\caption{$\Phi^{\ast}A$ }
	\label{fig:pixels2u}
\end{figure}

Since

\begin{enumerate}
\item[$(\mathcal{A}G_{1})$] For all $x_{ij},x_{kl}\in A$, $x_{ij}\cdot x_{kl}\in\Phi
{}^{\ast}A$,

\item[$(\mathcal{A}G_{2})$] For all $x_{ij},x_{kl},x_{mn}\in A$, $\left(  x_{ij}\cdot
x_{kl}\right)  \cdot x_{mn}=x_{ij}\cdot\left(  x_{kl}\cdot x_{mn}\right)  $
property\ holds in $\Phi{}^{\ast}A$,

\item[$(\mathcal{A}G_{3})$] For all $x_{ij}\in A$, $x_{ij}\cdot x_{55}=x_{55}\cdot
x_{ij}=x_{ij}$. Hence\emph{ }$x_{55}\in\Phi{}^{\ast}A$ is a
\textit{approximately identity element} of $A$,
\end{enumerate}

are satisfied, the subimage $A$ of the image $X$\ is indeed a descriptive approximately monoid in descriptive proximity space $\left(  X,\delta_{\Phi} \right)  $ with binary
operation \textquotedblleft\ $\cdot$ \textquotedblright. Also, since
$x_{ij}\cdot x_{kl}=x_{kl}\cdot x_{ij}$, for all $x_{ij},x_{kl}\in A$ property
holds in $\Phi{}^{\ast}A$, $A$ is a commutative descriptive approximately monoid.
\end{example}

\begin{example}
Let $X$ be a digital image endowed with descriptive proximity relation
$\delta_{\Phi}$ and consists of $36$ pixels as in Fig.\ref{fig:pixels3}.

\begin{figure}[!ht]
	\centering
		\includegraphics[scale=0.3]{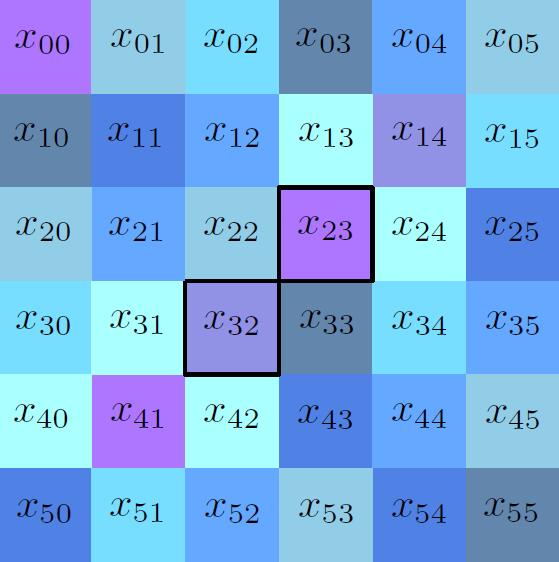}
	\caption{Digital image}
	\label{fig:pixels3}
\end{figure}

A pixel $x_{ij}$ is an element at position $\left(  i,j\right)  $ (row and
column) in digital image $X$. Let $\phi$ be a probe function that represent
RGB\ colour of each pixel are given in Table 2.
\footnotesize{
\[%
\begin{tabular}
[c]{l|ccc}
& $Red$ & $Green$ & $Blue$\\ \hline
$x_{00}$ & $174$ & $117$ & $255$\\
$x_{01}$ & $145$ & $205$ & $230$\\
$x_{02}$ & $120$ & $222$ & $255$\\
$x_{03}$ & $98$ & $134$ & $172$\\
$x_{04}$ & $100$ & $168$ & $255$\\
$x_{05}$ & $145$ & $205$ & $230$\\
$x_{10}$ & $98$ & $134$ & $172$\\
$x_{11}$ & $80$ & $130$ & $230$\\
$x_{12}$ & $100$ & $168$ & $255$\\
$x_{13}$ & $170$ & $255$ & $255$\\
$x_{14}$ & $145$ & $145$ & $230$\\
$x_{15}$ & $120$ & $222$ & $255$%
\end{tabular}
\
\begin{tabular}
[c]{l|ccc}
& $Red$ & $Green$ & $Blue$\\ \hline
$x_{20}$ & $145$ & $205$ & $230$\\
$x_{21}$ & $100$ & $168$ & $255$\\
$x_{22}$ & $145$ & $205$ & $230$\\
$x_{23}$ & $174$ & $117$ & $255$\\
$x_{24}$ & $170$ & $255$ & $255$\\
$x_{25}$ & $80$ & $130$ & $230$\\
$x_{30}$ & $120$ & $222$ & $255$\\
$x_{31}$ & $170$ & $255$ & $255$\\
$x_{32}$ & $145$ & $145$ & $230$\\
$x_{33}$ & $98$ & $134$ & $172$\\
$x_{34}$ & $120$ & $222$ & $255$\\
$x_{35}$ & $100$ & $168$ & $255$%
\end{tabular}
\text{\ }%
\begin{tabular}
[c]{l|ccc}
& $Red$ & $Green$ & $Blue$\\ \hline
$x_{40}$ & $170$ & $255$ & $255$\\
$x_{41}$ & $174$ & $117$ & $255$\\
$x_{42}$ & $170$ & $255$ & $255$\\
$x_{43}$ & $80$ & $130$ & $230$\\
$x_{44}$ & $100$ & $168$ & $255$\\
$x_{45}$ & $145$ & $205$ & $230$\\
$x_{50}$ & $80$ & $130$ & $230$\\
$x_{51}$ & $120$ & $222$ & $255$\\
$x_{52}$ & $100$ & $168$ & $255$\\
$x_{53}$ & $145$ & $205$ & $230$\\
$x_{54}$ & $80$ & $130$ & $230$\\
$x_{55}$ & $98$ & $134$ & $172$%
\end{tabular}
\
\]
}
\normalsize
\begin{center}
Table 2.
\end{center}
Let
\[%
\begin{tabular}
[c]{ll}%
$\odot:X\times X$ & $\longrightarrow X$\\
\multicolumn{1}{r}{$\left(  x_{ij},x_{kl}\right)  $} & $\longmapsto
x_{ij}\odot x_{kl}=x_{pr},$ $i+k=p$ $\operatorname{mod}\left(  5\right)  $ and
$j+l=r$ $\operatorname{mod}\left(  5\right)  $%
\end{tabular}
\
\]

be a binary operation on $X$ and $B=\left \{  x_{23},x_{32}\right \}  $ a
subimage (subset) of $X$.

We can compute the descriptively upper approximation of $B$ by using the
Definition \ref{Df4.1}. $\Phi^{\ast}B=\{x_{ij}\in X\mid x_{ij} \delta_{\phi} B  \}$, where $\mathcal{Q}(B)=\{ \phi(x_{ij})\mid
x_{ij}\in B\}$. Then  $\phi\left(  x_{ij}\right)
\cap \mathcal{Q}\left(  B\right)\neq \emptyset$\ such that $x_{ij}\in X$. From Table 2, we obtain%

\begin{tabular}
[c]{ll}%
$\mathcal{Q}(B)$ & $=\left \{  \phi \left(  x_{23}\right)  ,\phi \left(
x_{32}\right)  \right \}  $\\
& $=\left \{  \left(  174,117,255\right)  ,\left(  145,145,230\right)
\right \}  $.
\end{tabular}

Hence we get $\Phi^{\ast}B=\left \{  x_{23},x_{41},x_{00},x_{32},x_{14}%
\right \}  $.

Since

\begin{enumerate}
\item[$(\mathcal{A}G_{1})$] For all $x_{ij},x_{kl}\in A$, $x_{ij}\odot x_{kl}\in \Phi
{}^{\ast}B$,

\item[$(\mathcal{A}G_{2})$] For all $x_{ij},x_{kl},x_{mn}\in A$, $\left(  x_{ij}\odot
x_{kl}\right)  \odot x_{mn}=x_{ij}\odot \left(  x_{kl}\odot x_{mn}\right)  $
property\ holds in $\Phi{}^{\ast}B$,

\item[$(\mathcal{A}G_{3})$] For all $x_{ij}\in A$, $x_{ij}\odot x_{00}=x_{00}\odot
x_{ij}=x_{ij}$. Hence\emph{ }$x_{00}\in \Phi{}^{\ast}B$ is a
\textit{approximately identity element} of $B$,

\item[$(\mathcal{A}G_{4})$] Since $x_{32}\odot x_{23}=x_{23}\odot x_{32}=x_{00}$,
$x_{32}^{-1}=$ $x_{23}$ and  $x_{23}^{-1}=$ $x_{32}$, i.e., $x_{32},x_{23}$
are \textit{inverses} of $x_{23},x_{32}$, respectively.
\end{enumerate}

are satisfied, the subimage $B$ of the image $X$\ is indeed a descriptive approximately group in descriptive proximity space $\left(  X,\delta_{\Phi}\right)  $
with binary operation \textquotedblleft \ $\odot$ \textquotedblright. Also,
since $x_{ij}\odot x_{kl}=x_{kl}\odot x_{ij}$, for all $x_{ij},x_{kl}\in B$
property holds in $\Phi{}^{\ast}B$, $B$ is a commutative descriptive approximately group.
\end{example}

\begin{proposition}
Let $\left(  X,\mathcal{R}_{\delta_{\Phi}}\right)  $ be descriptive proximal
relator space and $G\subseteq X$ a descriptive approximately group.

(1) There is one and only one approximately identity element in $G$.

(2) $\forall x\in G$, there is only one $y\in G$ such that $x\cdot y=y\cdot
x=e$; we denote it by $x^{-1}$.

(3) $\left(  x^{-1}\right)  ^{-1}=x$, for all $x\in G$.

(4) $\left(  x\cdot y\right)  ^{-1}=y^{-1}\cdot x^{-1}$, for all $x,y\in G$.
\end{proposition}

\begin{definition}
Let $G$ be a descriptive approximately group and $H$ a non-empty subset of
$G$. $H$\ is called a descriptive approximately subgroup of $G$ if $H$ is a
descriptive approximately group relative to the operation in $G$.
\end{definition}

There is only one guaranteed trivial descriptive approximately subgroup of
$G$, i.e., $G$ itself. Moreover, $\left\{  e\right\}  $ is a trivial
descriptive approximately subgroup of $G$ if and only if $e\in G$.

\begin{theorem}
\label{Th4.1}Let $G$ be a descriptive approximately group, $H$ a non-empty
subset of $G$ and $\Phi^{\ast}H$ a groupoid. $H$ is a descriptive
approximately subgroup of $G$ if and only if $x^{-1}\in H$ for all $x\in H$.
\end{theorem}

\begin{proof}
Suppose that $H$ is a descriptive approximately subgroup of $G$. Then $H$ is a
descriptive approximately group and so $x^{-1}\in H$ for all $x\in H$.
Conversely, suppose $x^{-1}\in H$ for all $x\in H$. From the hypothesis, since
$\Phi^{\ast}H$ is a groupoid and $H\subseteq G$, then closure and associative
properties hold in $\Phi^{\ast}H$. Also we get $x\cdot x^{-1}=e\in\Phi^{\ast
}H$. Therefore $H$ is a descriptive approximately subgroup of $G$.
\end{proof}

\begin{theorem}
\label{Th4.2}Let $G$ be a descriptive approximately group and $H$ a non-empty
subset of $G$. $H$ is a descriptive approximately subgroup of $G$ if and only
if $Q\left(  H\right)  =Q\left(  G\right)  $.
\end{theorem}

\begin{proof}
It is obvious from Definition \ref{Df4.1}.
\end{proof}

\begin{theorem}
\label{Th4.3}Let $H_{1}$ and $H_{2}$ be two descriptive approximately
subgroups of descriptive approximately group $G$ and $\Phi^{\ast}H_{1}$,
$\Phi^{\ast}H_{2}$ groupoids. If%
\[
\left(  \Phi^{\ast}H_{1}\right)  \cap\left(  \Phi^{\ast}H_{2}\right)
=\Phi^{\ast}\left(  H_{1}\cap H_{2}\right)  \text{,}%
\]
then $H_{1}\cap H_{2}$ is a descriptive approximately subgroup of $G$.
\end{theorem}

\begin{proof}
Let $H_{1}$ and $H_{2}$ be two descriptive approximately subgroups of $G$. It
is obvious that $H_{1}\cap H_{2}\subset G$. Since $\Phi^{\ast}H_{1}$,
$\Phi^{\ast}H_{2}$ are groupoids and $\left(  \Phi^{\ast}H_{1}\right)
\cap\left(  \Phi^{\ast}H_{2}\right)  =\Phi^{\ast}\left(  H_{1}\cap
H_{2}\right)  $, $\Phi^{\ast}\left(  H_{1}\cap H_{2}\right)  $ is a groupoid.
Consider $x\in H_{1}\cap H_{2}$. Since $H_{1}$ and $H_{2}$ be two descriptive
approximately subgroups, we have $x^{-1}\in H_{1}$, $x^{-1}\in H_{2}$, i.e.,
$x^{-1}\in H_{1}\cap H_{2}$. Consequently, from Theorem \ref{Th4.1} $H_{1}\cap
H_{2}$ is a descriptive approximately subgroup of $G$.
\end{proof}

\begin{corollary}
In a descriptive proximity space, every proximal groups are descriptive
approximately groups.
\end{corollary}

\bibliographystyle{elsarticle-num}
\bibliography{NSrefs}

\end{document}